\documentclass[11pt]{article}

\usepackage[utf8]{inputenc}
\usepackage[english]{babel}
\usepackage{amsmath,amsthm,amssymb,amsfonts}
\usepackage{mathrsfs}
\usepackage{hyperref}

\newtheorem{theorem}{Theorem}[section]
\newtheorem{proposition}[theorem]{Proposition}
\newtheorem{lemma}[theorem]{Lemma}
\newtheorem{corollary}[theorem]{Corollary}
\newtheorem{question}[theorem]{Question}
\theoremstyle{definition}
\newtheorem{definition}[theorem]{Definition}
\newtheorem{remark}[theorem]{Remark}

\theoremstyle{plain}
\newtheorem*{thmA}{Theorem A}
\newtheorem*{thmB}{Theorem B}
\newtheorem*{thmC}{Theorem C}

\newcommand{\I}{\mathcal{I}}
\newcommand{\J}{\mathcal{J}}
\newcommand{\K}{\mathcal{K}}
\newcommand{\R}{\mathcal{R}}
\newcommand{\Ipos}{\mathcal{I}^{+}}
\newcommand{\EDfin}{\mathcal{ED}_{\mathrm{fin}}}
\newcommand{\nwd}{\mathrm{nwd}}
\newcommand{\conv}{\mathrm{conv}}
\newcommand{\leK}{\le_{K}}
\newcommand{\nleK}{\nleq_{K}}

\newcommand{\leKB}{\le_{KB}}
\newcommand{\Q}{\mathbb{Q}}
\newcommand{\QI}{\mathbb{Q}_{\I}}

\title{Reductions and necessary conditions for tall Borel Ramsey ideals}
\author{Jos\'e de Jes\'us Pelayo G\'omez}
\date{July 13, 2026}

\begin{document}
\maketitle

\begin{abstract}
An ideal $\I$ on $\omega$ has the \emph{Ramsey property} if $\Ipos\to(\Ipos)^2_2$: every
$2$-colouring of the pairs of an $\I$-positive set has an $\I$-positive homogeneous
subset. Whether a tall \emph{Borel} ideal can have the Ramsey property is an open question
of Hru\v{s}\'ak, Meza-Alc\'antara, Th\"ummel and Uzc\'ategui \cite{HMTU2017}; a coanalytic
example exists in ZFC, so a negative answer must use definability essentially. Our main
theorem, a synthesis of the results of the paper, states that a tall Borel Ramsey ideal
admits \emph{no countable local reading}. Below every positive set, such an ideal is not a
countable intersection of topologically represented (or tall analytic $P$-) ideals, and
its quotient has no countable dense subset. Moreover, every quotient
name for a new real has uncountable width, the colouring witnessing non-selectivity of the
generic ultrafilter is never read continuously on a positive condition, and hereditary
tall subfamilies saturate every finite window of barrier dimensions coherently but never
all dimensions at once. We prove separately that a weakly selective $q^+$ ideal admits no
positive $\EDfin$-carrier. The converse question --- must Borelness force a properness-like
countable reading on some positive condition? --- is stated in three precise forms with
proved consequences: two of them would refute tall Borel Ramsey ideals outright, the third
the strictly weaker Nash--Williams class. The main question remains open.
\end{abstract}


\section{Introduction}

We work with ideals $\I\subseteq\mathcal P(\omega)$ containing the finite sets, proper, and
identified with subsets of Cantor space via characteristic functions, so that descriptive
notions ($F_\sigma$, Borel, analytic, coanalytic) apply. An ideal is \emph{tall} if every
infinite set has an infinite subset in $\I$. We write $\Ipos=\mathcal P(\omega)\setminus\I$
and $\I\upharpoonright X=\{A:A\cap X\in\I\}$. Katětov reducibility $\I\leK\J$ means there is
$f$ with $f^{-1}[A]\in\J$ for all $A\in\I$; $\leKB$ is its finite-to-one refinement.

\begin{definition}
$\I$ has the \emph{Ramsey property} (here: positive/hereditary Ramsey) if
$\Ipos\to(\Ipos)^2_2$: for every $X\in\Ipos$ and every $c:[X]^2\to2$ there is
$Y\subseteq X$, $Y\in\Ipos$, with $c\upharpoonright[Y]^2$ constant. It has the
\emph{ordinary Ramsey property}, $\omega\to(\Ipos)^2_2$, if this holds for $X=\omega$ only.
\end{definition}

The basic Katětov characterisation, with $\R$ the random-graph ideal, is
$\I$ Ramsey $\iff$ $\R\nleK\I\upharpoonright X$ for all $X\in\Ipos$
\cite{HMTU2017}; see also \cite{CDU2026}.

\paragraph{The problem.} The following is Question~6.1 of \cite{HMTU2017} (also
\cite[Q.5.19]{HrusakKatetov2017}, the survey \cite{HrusakSurvey}, and restated as open in
\cite{CDU2026}).

\begin{question}[\cite{HMTU2017}]\label{q:main}
Is there a tall Borel ideal $\I$ with $\Ipos\to(\Ipos)^2_2$?
\end{question}

\paragraph{Antecedents and landscape.}
\begin{itemize}
\item No tall $F_\sigma$ ideal is Ramsey (Hru\v{s}\'ak's theorem
$\EDfin\leK\I\upharpoonright X$ locally, plus $\R\leK\EDfin$) \cite{HMTU2017,HrusakKatetov2017}.
\item No tall analytic $P$-ideal is Ramsey (Solecki's $\mathrm{Exh}(\varphi)$ form and a
submeasure block argument) \cite{HMTU2017}.
\item There is a coanalytic tall Ramsey ideal in ZFC (\cite[Thm~4.8]{HMTU2017}, modelled on
\cite[Ex.~2.8]{Farah1998}; the authors explicitly leave open, in their Remark~4.9, whether
such an ideal can be Borel); hence Question~\ref{q:main} genuinely separates Borel from
coanalytic, and cannot be settled by any property shared with that example.
\item Every tall analytic ideal has a tall $F_\sigma$ subideal \cite{GrebikVidnyanszky2023};
there is no minimal tall Borel ideal in $\leK$ \cite{GrebikHrusak2020}.
\item \cite{CDU2026} gives the current combinatorial toolkit: Ramsey collapses all finite
dimensions and colour counts (their Thm~3.13, Fact~3.2); simultaneous almost-homogeneity
(their Prop.~3.15, Thm~3.16); a canonical (Erd\H{o}s--Rado) form (their Thm~3.30); and a
ladder \emph{selective $\Rightarrow$ semiselective $\Rightarrow$ Galvin $\Rightarrow$
Nash--Williams $\Rightarrow$ Ramsey}, with no analytic tall Galvin ideal (their Prop.~4.12).
Their tall Nash--Williams non-semiselective ideal (Thm~4.44) is a \emph{construction} that
provably cannot be made coanalytic, because the set of fronts is $\Pi^1_1$-complete
(Prop.~4.45) and every analytic family of fronts is rank-bounded (Prop.~4.46). The
random-hypergraph reformulation via $\R^n_k$ is \cite[Prop.~3.4]{PelayoGomez}.
\end{itemize}

\paragraph{Main results.} The main theorem asserts that a tall Borel Ramsey ideal escapes
every \emph{countable local reading} of its structure, in four senses simultaneously. It
is a synthesis, proved by assembling the results of
Sections~\ref{sec:countable}--\ref{sec:core}; the full statement, with the exact
hypotheses, is Theorem~\ref{thm:master}.

\begin{thmA}[= Theorem~\ref{thm:master}]
Let $\I$ be a tall Borel Ramsey ideal and let $X\in\Ipos$. Then:
\begin{enumerate}
\item[\textup{(a)}] \textup{(representations)} $\I\upharpoonright X$ is not the
intersection of countably many ideals each failing the ordinary Ramsey property --- in
particular, of countably many topologically represented ideals --- nor of countably many
tall analytic $P$-ideals; in particular the Kwela--Sabok
representation of $\I\upharpoonright X$ has no countable subfamily with the same
intersection.
\item[\textup{(b)}] \textup{(quotient topology)} $\mathcal P(\omega)/\I$ has no countable
dense subset below $X$.
\item[\textup{(c)}] \textup{(forcing names)} every $\QI$-name for a new real has
uncountable width below $X$, and the colouring witnessing non-selectivity of the generic
ultrafilter is not read continuously into the ground model on any positive
$A\subseteq X$.
\item[\textup{(d)}] \textup{(barriers)} for every closed hereditary tall
$\mathcal K\subseteq\I$ and all dimensions $n_1<\dots<n_m$ there is one positive
$H\subseteq X$ with $[H]^{n_i}\subseteq\mathcal K$ for every $i\le m$, yet no positive
$H\subseteq X$ satisfies $[H]^n\subseteq\mathcal K$ for all $n$ simultaneously.
\end{enumerate}
\end{thmA}

In one sentence: no countable amount of definable data --- representing ideals, dense sets
of conditions, deciding antichains, barrier dimensions --- reads the positivity of $\I$
below any positive set. The coanalytic example shows that $\Pi^1_1$ codes genuinely enjoy
this freedom; Question~\ref{q:main} asks whether a $\Delta^1_1$ code can afford it.

Two further results are of independent interest. Theorem~B is the sharpest input to
face~(a); Theorem~C is a necessary condition of a different flavour and closes the
carrier route.

\begin{thmB}[= Theorems~\ref{thm:int-sub} and~\ref{thm:int-nonramsey}]
No countable intersection of tall analytic $P$-ideals, or of ideals failing the ordinary
Ramsey property --- in particular, of topologically represented ideals --- is Ramsey.
\end{thmB}

\begin{thmC}[= Theorem~\ref{thm:carrier}]
If $\I$ is weakly selective and $q^+$ (both necessary for Ramsey), then for every
$Y\in\Ipos$ no map witnesses $\EDfin\leK\I\upharpoonright Y$; in particular no subideal of
a Ramsey candidate is locally above $\EDfin$ on a positive set, so no carrier program in
the style of Greb\'ik--Vidny\'anszky can produce an $\I$-positive $\EDfin$-carrier.
\end{thmC}

The paper ends with the converse of Theorem~A, which the results above isolate as the
remaining content of Question~\ref{q:main}: \emph{must Borelness of the code force a
countable, properness-like reading of the quotient on some positive condition?} We state
it in three precise forms, each with its consequence proved. The principles \textup{(DW)}
(countable width of some new-real name) and \textup{(RC)} (continuous reading of
colourings) each imply a negative answer to Question~\ref{q:main}
(Corollaries~\ref{cor:DW} and~\ref{cor:RC}); a single Schreier-barrier colouring
(Question~\ref{q:schreier}) implies that there is no tall Borel Nash--Williams ideal
(Proposition~\ref{prop:schreier-suff}), via a saturation lemma proved here
(Lemma~\ref{lem:saturation}). Whether there is a tall Borel Nash--Williams ideal is itself
open and strictly weaker than Question~\ref{q:main}; it is not literally asked in
\cite{CDU2026} --- their Questions~4.47 and~4.48 concern the nondefinable separations ---
and by \cite[Prop.~4.34]{CDU2026} the Schreier form is the stronger target, with their
Question~4.39 as calibration. A selector obstruction (Proposition~\ref{prop:selector})
shows the required colouring cannot come from Borelness and tallness alone.

\paragraph{Organization.} Section~\ref{sec:prelim} collects the known results and the two
localization tools used throughout: notation, the Katětov characterisation, the standard
necessary conditions. \emph{From Section~\ref{sec:countable} on, every numbered statement
is due to the present paper unless an attribution appears in its header}; results quoted
from the literature carry the reference with the statement, and are never reproved. Where
a new statement repackages a known fact, or where we could not exclude folklore status,
the surrounding text says so explicitly. Section~\ref{sec:countable} proves Theorem~B, the
representation frontier, and a topological width bound. Section~\ref{sec:reduced} gives a
normal form below $\nwd(\Q)$ and isolates the limit spectrum. Section~\ref{sec:necessary}
proves Theorem~C and face~(b) of Theorem~A. Section~\ref{sec:ladder} develops the barrier
saturation machinery behind face~(d) and the Schreier question, together with the
principle that obstructions must see the ideal. Section~\ref{sec:core} proves face~(c) and
the principles (DW) and (RC). Section~\ref{sec:targets} assembles Theorem~A, records a
forcing-absoluteness transfer, summarizes the constraints established, and states the
three remaining questions. Appendix~\ref{app:completion} explains, in Boolean language,
why the real added by the quotient of a candidate lies in the completion
$\mathrm{RO}(\mathbb B)\setminus\mathbb B$, and why low complexity of the added real
cannot open (DW).

\section{Preliminaries and known constraints}\label{sec:prelim}

We use $\nwd$ (nowhere dense subsets of $\Q$), $\conv$ (subsets of a convergent-sequence
space with countable-compact closure), $\mathcal{ED}$ on $\omega\times\omega$ (sets whose
columns are eventually of bounded size), $\EDfin=\mathcal{ED}\upharpoonright\Delta$ on
$\Delta=\{(n,m):m\le n\}$, and $\R$. Recall $\R\leK\EDfin$, $\R\leK\mathcal{ED}$, and
$\R\leK\conv\leK\mathrm{Fin}\times\mathrm{Fin}$.

An ideal $\J$ has a \emph{topological representation} \cite{KwelaSabok2015} (a notion rooted
in Solecki's analysis of covering by closed sets \cite{Solecki1994}) if
$A\in\J\iff\overline{e[A]}\in\K$ for a bijection $e:\omega\to D$, $D$ dense in a compact
metric $K$, and $\K$ a $\sigma$-ideal of compact sets. A \emph{barrier} / \emph{front} on
$\omega$ and the associated $B$-Ramsey, Nash--Williams and Galvin properties are as in
\cite{CDU2026}. The Schreier barrier is $\mathscr S=\{s:|s|=\min(s)+1\}$. We write
$\QI=\mathcal P(\omega)/\I$ for the quotient forcing, whose conditions are the
$\I$-positive sets ordered by $A\le B\iff A\setminus B\in\I$, with
$A\perp B\iff A\cap B\in\I$; for Borel $\I$ this is a Borel (Suslin) forcing.

\begin{theorem}[\cite{HMTU2017}]\label{thm:kat}
$\I$ is Ramsey iff $\R\nleK\I\upharpoonright X$ for every $X\in\Ipos$. Consequently the
Ramsey property is inherited by positive restrictions and (\cite[Prop.~3.4]{PelayoGomez},
\cite[Thm~3.13]{CDU2026}) is equivalent to $\R^n_k\nleK\I\upharpoonright X$ for all
$X\in\Ipos$ and all $n,k\ge2$.
\end{theorem}

\begin{theorem}[standard necessary conditions; assembled from
\cite{HMTU2017,HrusakKatetov2017}]\label{thm:standard}
A tall Borel Ramsey ideal $\I$ satisfies, on \emph{every} positive restriction: weak
selectivity and $q^+$; failure of $p^+$; $p^+_{\mathrm{tower}}$;
$(\omega,2)$-distributive quotient; and decomposability (hence $\I$ is hereditarily
decomposable). Moreover $\I\leK\nwd$.
\end{theorem}

Every ingredient of the Kat\v{e}tov bound $\I\leK\nwd$ is likewise known; since we could
not locate the bound stated explicitly in the literature, we include the two-line
assembly, claiming no novelty for it.

\begin{proof}[Proof of the bound $\I\leK\nwd$]
By Hru\v{s}\'ak's Category Dichotomy \cite{HrusakKatetov2017}, for every Borel ideal $\I$,
either $\I\leK\nwd$, or there is $X\in\Ipos$ with $\mathcal{ED}\leK\I\upharpoonright X$.
Since $\R\leK\mathcal{ED}$, the second alternative gives $\R\leK\I\upharpoonright X$,
which contradicts the Ramsey property by Theorem~\ref{thm:kat}.
\end{proof}

\begin{remark}[a game-theoretic route to the same bound]\label{rem:G3}
The bound is not an artifact of the dichotomy. Let $G_3(\I)$ be the game in which, at
stage $n$, Player~I plays a set $A_n\in\I$ and Player~II answers with a point
$x_n\in\omega\setminus A_n$, Player~I winning iff $\{x_n:n\in\omega\}\in\I$. If Player~II
has a winning strategy in $G_3(\I)$, then $\I\leK\nwd$, by the game characterization of
Hru\v{s}\'ak--Meza-Alc\'antara (see \cite{HrusakSurvey}). An elementary verification,
using only the $(\omega,2)$-distributivity and $q^+$ items of Theorem~\ref{thm:standard},
shows that for a Ramsey $\I$ Player~I has no winning strategy in
$G_3(\I\upharpoonright X)$ for any $X\in\Ipos$; for Borel $\I$ the game is determined
\cite{Martin1975}, so Player~II wins. Thus the bound is already forced by the
game-theoretic consequences of the Ramsey property alone. We omit the verification, as
the dichotomy already gives the bound.
\end{remark}

Two localization tools close this section. The first is likely folklore in the
Katětov--Blass order, but we could not locate it in the form we need; the second is
folklore, and we will include its three-line proof for completeness.

\begin{proposition}[injectivization]\label{prop:inj}
Let $\I$ be an ideal on $E$, $\J$ a weakly selective ideal on $D$, and $f:D\to E$ a witness
of $\I\leK\J$. Then there is $Y\in\J^+$ on which $f$ is injective; with $T=f[Y]\in\Ipos$, the
transported copy of $\I\upharpoonright T$ is contained in $\J\upharpoonright Y$.
\end{proposition}

\begin{proof}
Each nonempty fiber $f^{-1}(\{e\})$ lies in $\J$, because $\{e\}\in\I$ and $f$ witnesses
$\I\leK\J$. Apply weak selectivity of $\J$ (every partition of a $\J$-positive set into
$\J$-small pieces has a $\J$-positive selector) to the partition of $D$ into the nonempty
fibers of $f$: there is $Y\in\J^+$ meeting each fiber in at most one point, so
$f\upharpoonright Y$ is injective. If $T=f[Y]$ were in $\I$, then
$Y\subseteq f^{-1}[T]\in\J$, a contradiction; hence $T\in\Ipos$. Finally, for
$A\in\I\upharpoonright T$ we have
$(f\upharpoonright Y)^{-1}[A]=f^{-1}[A\cap T]\cap Y\in\J$, so the copy of
$\I\upharpoonright T$ transported to $Y$ is contained in $\J\upharpoonright Y$.
\end{proof}

\begin{lemma}[$\sigma$-compact splitter; folklore]\label{lem:splitter}
Every $\sigma$-compact $\K\subseteq[\omega]^\omega$ has a common splitter.
\end{lemma}

\begin{proof}
Write $\K=\bigcup_n K_n$ with each $K_n$ compact in $[\omega]^\omega$. For $k\in\omega$
the function $q_k(B)=\min(B\setminus[0,k])$ is continuous and integer-valued on
$[\omega]^\omega$, hence bounded on each $K_n$. Choose $k_0<k_1<\cdots$ recursively so
that
\[
 k_{j+1}\ \ge\ \max\{q_{k_j}(B):B\in K_0\cup\dots\cup K_j\}.
\]
Then every $B\in K_n$ meets every interval $(k_j,k_{j+1}]$ with $j\ge n$, so the union of
the even-indexed intervals splits every member of $\K$.
\end{proof}

\section{Countable intersections and the representation
frontier}\label{sec:countable}

The first group of results shows that the Ramsey property never arises as a countable
intersection of represented pieces. The two theorems are independent; neither subsumes
the other.

\begin{theorem}[submeasure form]\label{thm:int-sub}
No countable intersection of tall analytic $P$-ideals has the Ramsey property. In fact,
for each such intersection $\I$ and each $X\in\Ipos$ there is a positive $Y\subseteq X$
with $\EDfin\leK\I\upharpoonright Y$.
\end{theorem}

\begin{proof}
Write $\I=\bigcap_{m\in\omega}\J_m$ with each $\J_m$ tall, and fix lower semicontinuous
submeasures $\varphi_m$ with $\J_m=\mathrm{Exh}(\varphi_m)$, by Solecki's theorem. Let
$X\in\Ipos$ and fix $m_*$ with $X\notin\J_{m_*}$; the tail masses
$\varphi_{m_*}(X\setminus N)$ decrease to a positive limit, so we may fix $\varepsilon>0$
with $\varphi_{m_*}(X\setminus N)>\varepsilon$ for every $N$.

For $m\in\omega$ and $\delta>0$ the set
$H_{m,\delta}=\{n\in\omega:\varphi_m(\{n\})\ge\delta\}$ is finite: if it were infinite,
tallness of $\J_m$ would give an infinite $Z\subseteq H_{m,\delta}$ with $Z\in\J_m$; but
every tail $Z\setminus n$ contains a point of mass $\ge\delta$, so
$\varphi_m(Z\setminus n)\ge\delta$ for all $n$ and $Z\notin\mathrm{Exh}(\varphi_m)$, a
contradiction.

We recursively choose finite sets $b_j\subseteq X$ so that:
\begin{itemize}
\item $\max b_j<\min b_{j+1}$;
\item $\varphi_{m_*}(b_j)>\varepsilon$;
\item for every $m\le j$ and every $n\in b_j$,
$\varphi_m(\{n\})<2^{-m-1}/(j+1)^2$.
\end{itemize}
Assume $b_i$ has been chosen for $i<j$. The set
$F_j=\bigcup_{m\le j}H_{m,\,2^{-m-1}/(j+1)^2}$ is finite. Choose $N$ above $F_j$ and above
all previous blocks; since $\varphi_{m_*}(X\setminus N)>\varepsilon$, lower
semicontinuity yields a finite $b_j\subseteq X\setminus N$ with
$\varphi_{m_*}(b_j)>\varepsilon$, and $b_j\cap F_j=\emptyset$ gives the singleton bounds.

Set $Y=\bigcup_j b_j\subseteq X$. For any block $b_j$ beyond $n$ we have
$\varphi_{m_*}(Y\setminus n)\ge\varphi_{m_*}(b_j)>\varepsilon$, so
$Y\notin\mathrm{Exh}(\varphi_{m_*})=\J_{m_*}$, hence $Y\in\Ipos$.

Choose pairwise distinct integers $r_j$ with $r_j+1\ge|b_j|$ and injections
$f_j:b_j\hookrightarrow\{(r_j,t):t\le r_j\}$, the $r_j$-th column of $\Delta$. Let
$f=\bigcup_jf_j:Y\to\Delta$, extended arbitrarily to $\omega\setminus Y$ (the extension
is irrelevant, as membership in $\I\upharpoonright Y$ only depends on the trace on $Y$).

We claim $f$ witnesses $\EDfin\leK\I\upharpoonright Y$. Let $A\in\EDfin$ and choose $k$
and $n_0$ such that $|A\cap\mathrm{col}_n|\le k$ for all $n\ge n_0$, where
$\mathrm{col}_n=\{(n,m):m\le n\}$. The finitely many
exceptional columns pull back into finitely many blocks, hence into a finite subset of
$Y$, which lies in every $\J_m$; so we may assume all columns of $A$ have size $\le k$.
Then $|f^{-1}[A]\cap b_j|\le k$ for every $j$, and for $m\le j$,
\[
 \varphi_m\bigl(f^{-1}[A]\cap b_j\bigr)
 \le\sum_{n\in f^{-1}[A]\cap b_j}\varphi_m(\{n\})
 < k\cdot\frac{2^{-m-1}}{(j+1)^2}.
\]
Hence, for every $N\ge m$, $\sigma$-subadditivity gives
\[
 \varphi_m\Bigl(f^{-1}[A]\cap Y\setminus\bigcup_{j<N}b_j\Bigr)
 \le\sum_{j\ge N}k\cdot\frac{2^{-m-1}}{(j+1)^2},
\]
which tends to $0$ as $N\to\infty$; that is,
$f^{-1}[A]\cap Y\in\mathrm{Exh}(\varphi_m)$. As $m$ was arbitrary,
$f^{-1}[A]\in\I\upharpoonright Y$.

Finally, $\R\leK\EDfin$ (\S\ref{sec:prelim}), so Theorem~\ref{thm:kat} applied to the
positive $Y$ shows $\I$ is not Ramsey.
\end{proof}

\begin{remark}[the $F_\sigma$ analogue]\label{rem:fsigma-open}
A finite intersection of tall $F_\sigma$ ideals is a tall $F_\sigma$ ideal, so it is not
Ramsey by Hru\v{s}\'ak's theorem quoted in the introduction. Whether a countably
infinite intersection of tall $F_\sigma$ ideals can be Ramsey we do not know. The proof
above breaks for $\mathrm{Fin}(\varphi_m)$-presentations at the finiteness of
$H_{m,\delta}$, and no normalization restores it: $\EDfin$ is a tall $F_\sigma$ ideal
admitting no lower semicontinuous presentation with vanishing singleton masses (if
$\varphi(\{x\})\to0$, then for each $j$ one may pick $j$ points of mass $<2^{-j}/j$
inside a single column beyond all previously used ones --- only finitely many points have
mass $\ge2^{-j}/j$, and late columns are long; the union $S$ of these chunks has
$\varphi(S)\le\sum_j2^{-j}<\infty$ by $\sigma$-subadditivity, hence $S$ would lie in
$\mathrm{Fin}(\varphi)=\EDfin$, yet its columns are unbounded).
\end{remark}

\begin{lemma}\label{lem:toprep-nonramsey}
Every ideal with a topological representation fails the ordinary Ramsey property.
\end{lemma}

\begin{proof}
Let $\J$ be topologically represented. By \cite{KwelaSabok2015} the compact space of a
representation may be taken to be $2^\omega$; fix such a representation, with bijection
$e:\omega\to D\subseteq2^\omega$ and $\sigma$-ideal of compact sets $\K$. Let $<_{\mathrm{lex}}$ be the lexicographic
order on $2^\omega$ and colour $\{i,j\}$ with $i<j$ according to whether
$e(i)<_{\mathrm{lex}}e(j)$. An infinite homogeneous $H$ enumerates, in its natural order,
a lexicographically monotone sequence $(x_p)_p$ in $2^\omega$; such a sequence converges
in $2^\omega$: $x_p(0)$ is monotone in $p$, hence eventually constant, and once the first
$i$ coordinates have stabilized the sequence is $<_{\mathrm{lex}}$-monotone in the
remaining coordinates, so $x_p(i)$ also stabilizes. Thus $\overline{e[H]}$ is a countable
compact set. Since $\J$ contains the singletons, $\K$ contains every singleton of $D$,
and being a $\sigma$-ideal of compact sets it contains every countable compact set;
hence $\overline{e[H]}\in\K$, i.e.\ $H\in\J$. So no colouring-homogeneous set is
$\J$-positive.
\end{proof}

\begin{theorem}[non-Ramsey($\omega$) form]\label{thm:int-nonramsey}
If $\I=\bigcap_m\J_m$ and each $\J_m$ fails ordinary Ramsey$(\omega)$, then $\I$ is not
Ramsey. In particular no countable intersection of ideals with topological representations is
Ramsey.
\end{theorem}

\begin{proof}
Suppose $\I$ were Ramsey. For each $m$ the failure of $\omega\to(\J_m^+)^2_2$ provides a
colouring $c_m:[\omega]^2\to2$ all of whose homogeneous sets lie in $\J_m$. By the
simultaneous almost-homogeneity theorem (\cite[Prop.~3.15]{CDU2026}) applied to the
Ramsey ideal $\I$, there is $H\in\Ipos$ almost homogeneous for every $c_m$: for each $m$
there is a finite $F_m\subseteq H$ with $c_m$ constant on $[H\setminus F_m]^2$. Then
$H\setminus F_m\in\J_m$, and ideals are invariant under finite changes, so $H\in\J_m$ for
every $m$; hence $H\in\bigcap_m\J_m=\I$, contradicting $H\in\Ipos$. The second statement
follows from Lemma~\ref{lem:toprep-nonramsey}.
\end{proof}

\begin{remark}[non-redundancy]\label{rem:nonredundant}
Theorem~\ref{thm:int-nonramsey} does not imply Theorem~\ref{thm:int-sub}: what a tall
analytic $P$-ideal provides is $\R\leK\J\upharpoonright X$ on a positive restriction, and
this does not give $\R\leK\J$, i.e.\ failure of ordinary Ramsey$(\omega)$. Conversely,
Theorem~\ref{thm:int-sub} requires the pieces to be $P$-ideals.
\end{remark}

\begin{corollary}[the representation frontier]\label{cor:frontier}
By Kwela--Sabok \cite{KwelaSabok2015}, every coanalytic weakly selective ideal is an
intersection of topologically represented ideals; the Ramsey property implies weak
selectivity, so this applies to every putative Borel example.
Theorem~\ref{thm:int-nonramsey} shows the representing intersection has no countable
subfamily with the same intersection. Hence either statement below settles
Question~\ref{q:main} negatively:
\begin{itemize}
\item every Borel weakly selective ideal is, on some positive restriction, a countable
intersection of topologically represented ideals; or
\item the Kwela--Sabok representation of a Borel ideal admits a countable cofinal subfamily
on some positive restriction.
\end{itemize}
The coanalytic example escapes precisely by having essential uncountable cofinality here.
\end{corollary}

The countable/analytic shadow of the frontier is a topological width bound. In
cardinal-invariant form the first half of the next theorem is folklore (reaping families
are unsplittable, hence non-$\sigma$-compact); the point is the Hurewicz-type
strengthening for analytic families.

\begin{theorem}[width is continuum]\label{thm:width}
Let $\mathcal N\subseteq[\omega]^\omega$ be a positive $\pi$-base mod finite for the coideal
of a proper ideal (every positive set almost contains a member). Then $\mathcal N$ is
unsplittable, hence uncountable; and if $\mathcal N$ is analytic it contains a copy of
$\omega^\omega$ closed in $[\omega]^\omega$.
\end{theorem}

\begin{proof}
First, $\mathcal N$ is unsplittable: given $Z\subseteq\omega$, at least one of $Z$,
$\omega\setminus Z$ is positive (the ideal is proper); that side almost contains some
$N\in\mathcal N$, and then $Z$ does not split $N$. Every countable family of infinite
sets has a common splitter (visit each member infinitely often, alternately placing and
withholding fresh points), so $\mathcal N$ is uncountable. By Lemma~\ref{lem:splitter}
no subfamily of a $\sigma$-compact family is unsplittable either, so $\mathcal N$ is not
contained in a $K_\sigma$ subset of $[\omega]^\omega$. If $\mathcal N$ is analytic, the
Hurewicz dichotomy for analytic sets (\cite[Thm~21.18]{Kechris1995}) says it is either
contained in a $K_\sigma$ set or contains a closed-in-the-ambient-space copy of
$\omega^\omega$; the first alternative has just been excluded.
\end{proof}

\begin{remark}
For the coanalytic example this proves its node/residual families contain closed copies of
$\omega^\omega$: the continuum width is forced, already at the level of arbitrary ideals. It
also explains why any ``countable width'' reduction is a category error: no ideal has a
countable positive $\pi$-base mod finite.
\end{remark}

\section{The reduced presentation below $\nwd$}\label{sec:reduced}

By Theorem~\ref{thm:standard}, a tall Borel Ramsey ideal satisfies $\I\leK\nwd$. We upgrade
this to an honest containment on a positive set.

\begin{corollary}[normal form: honestly below $\nwd$]\label{cor:normal}
Let $\I$ be a tall Borel Ramsey ideal. There are $T\in\Ipos$, a countable dense-in-itself
$S\cong\Q$, and a bijection $h:S\to T$ such that
$\overline{\I}=\{A\subseteq S:h[A]\in\I\}$ is a tall Borel Ramsey ideal with
$\overline{\I}\subseteq\nwd(S)$. Thus WLOG $\I\subseteq\nwd(\Q)$, and every somewhere dense
subset of $\Q$ is $\I$-positive.
\end{corollary}

\begin{proof}
By Theorem~\ref{thm:standard}, $\I\leK\nwd$; realize $\nwd$ on $\Q$ and fix a witness
$f:\Q\to\omega$. The ideal $\nwd(\Q)$ is weakly selective, so Proposition~\ref{prop:inj}
gives a somewhere dense $Y\subseteq\Q$ on which $f$ is injective. Choose a rational open
interval $U$ in which $Y$ is dense and put $S=Y\cap U$ and $h=f\upharpoonright S$. Then
$S$ is countable and dense in itself, hence homeomorphic to $\Q$ by Sierpi\'nski's
characterization, and for $A\subseteq S$, nowhere density in $S$, in $U$, and in $\Q$ all
coincide. Put $T=h[S]$. If $T\in\I$, then $S\subseteq f^{-1}[T]\in\nwd(\Q)$, contradicting
that $S$ is dense in $U$; so $T\in\Ipos$. If $A\subseteq S$ and $h[A]\in\I$, then
$A\subseteq f^{-1}[h[A]]\in\nwd(\Q)$, so $A\in\nwd(S)$: this is the containment
$\overline\I\subseteq\nwd(S)$. Finally, $\overline\I$ is the isomorphic copy of
$\I\upharpoonright T$ induced by the bijection $h:S\to T$, and tallness, Borelness and the
Ramsey property pass to positive restrictions (Theorem~\ref{thm:kat}) and to isomorphic
copies. Identifying $S$ with $\Q$ gives the normal form; the last clause restates
$\overline\I\subseteq\nwd(S)$.
\end{proof}

In the normal form, positivity concentrates on convergent sequences: by the $h$-FinBW
property of Ramsey ideals (\cite[Prop.~3.7(f)]{CDU2026}), if $\I\subseteq\nwd(D)$ with $D$
dense in a compact interval $K$, then every $X\in\Ipos$ contains a positive $C$ converging
in $K$. What is new is the global object this produces.

\begin{proposition}[the limit spectrum]\label{prop:conv}
Let $\I\subseteq\nwd(D)$ be a tall Borel Ramsey ideal, $D$ dense in a compact interval $K$.
The \emph{limit spectrum}
\[
   L_\I=\{z\in K:(\exists C\in\Ipos)\ C\to z\}
\]
is analytic and dense in $K$, hence \emph{countable or containing a perfect set}.
\end{proposition}

\begin{proof}
Since $\I$ is Borel, the relation ``$C\notin\I$ and $C$ converges to $z$'' is Borel in
$(C,z)$, so its projection $L_\I$ is analytic; an uncountable analytic set contains a
perfect set, whence the last clause. For density, let $U\subseteq K$ be open and nonempty
and choose open nonempty $V$ with $\overline V\subseteq U$. The set $D\cap V$ is somewhere
dense in $D$, hence $\I$-positive because $\I\subseteq\nwd(D)$. By the $h$-FinBW property
quoted above, $D\cap V$ contains a positive $C$ converging in $K$; its limit lies in
$\overline V\subseteq U$, so $L_\I\cap U\ne\emptyset$.
\end{proof}

\begin{remark}\label{rem:selfsimilar}
The construction is self-similar: $\I\upharpoonright C$ is again tall Borel Ramsey, so it
again admits its own reduced presentation below $\nwd$ under re-embedding of $C$ as $\Q$. Any
rank argument must break this self-similarity at the level of \emph{presentations}, not of
the ideal (see \S\ref{sec:targets}).
\end{remark}

\section{Two necessary conditions}\label{sec:necessary}

Each theorem of this section rules out an entire family of candidate constructions. The
$\leKB$ version of the first is close to the known $q^+$ characterizations; the content
here is the upgrade to \emph{arbitrary} Katětov witnesses (via
Proposition~\ref{prop:inj}), which is exactly what carrier programs require.

\begin{theorem}[no positive $\EDfin$-carrier]\label{thm:carrier}
If $\I$ is weakly selective and $q^+$ (both necessary for Ramsey, Theorem~\ref{thm:standard}),
then for \emph{every} $Y\in\Ipos$ no map witnesses $\EDfin\leK\I\upharpoonright Y$.
Consequently, for any subideal $\J\subseteq\I$ and any $Y\in\Ipos$,
$\EDfin\leK\J\upharpoonright Y$ is impossible: the Greb\'ik--Vidny\'anszky carrier program
cannot produce an $\I$-positive carrier.
\end{theorem}

\begin{proof}
Suppose $f:Y\to\Delta$ witnesses $\EDfin\leK\I\upharpoonright Y$. Every cell
$\{(n,m)\}$ is finite, hence in $\EDfin$, so every nonempty fiber $f^{-1}[\{(n,m)\}]$
meets $Y$ in an $\I$-small set: the fibers partition $Y$ into countably many
$\I$-small pieces. Weak selectivity below $Y$ gives a positive $S\subseteq Y$ meeting
each fiber in at most one point, so $f\upharpoonright S$ is injective. Column $n$ of
$\Delta$ has $n+1$ cells, so $S_n=S\cap f^{-1}[\{(n,m):m\le n\}]$ has at most $n+1$
elements: the sets $(S_n)_{n\in\omega}$ partition the positive set $S$ into finite
pieces. Now let $T$ be \emph{any} selector of $(S_n)_n$. Since $f\upharpoonright S$ is
injective and picks at most one point per column, $f[T]$ has at most one point in each
column of $\Delta$, so $f[T]\in\EDfin$, whence
$T\subseteq f^{-1}[f[T]]\cap Y\in\I$. Thus every selector of a partition of the positive
set $S$ into finite pieces is $\I$-small, contradicting $q^+$ (which grants such a
partition a positive selector).

For the consequence: if $\J\subseteq\I$, $Y\in\Ipos$ and $f$ witnessed
$\EDfin\leK\J\upharpoonright Y$, then, since $\J\upharpoonright Y\subseteq
\I\upharpoonright Y$, the same $f$ would witness $\EDfin\leK\I\upharpoonright Y$, which
the theorem forbids.
\end{proof}

\begin{theorem}[no countable dense set in the quotient]\label{thm:qdense}
If $\I$ is tall Borel Ramsey, then no positive restriction of $\mathcal P(\omega)/\I$ has a
countable dense subset.
\end{theorem}

\begin{proof}
Suppose the quotient below some $X\in\Ipos$ had a countable dense subset. The quotient is
atomless below $X$: if $C\subseteq X$ were an atom, then
$\mathcal U=\{A\subseteq C:C\setminus A\in\I\}$ would be a nonprincipal ultrafilter on
$C$ (nonprincipal because $\I$ contains the singletons) which is Borel
because $\I$ is --- impossible, since a nonprincipal ultrafilter lacks the Baire
property. The quotient order is separative, so a countable dense subset below $X$ is a
countable atomless separative partial order densely embedded in the quotient below $X$;
every such order is forcing equivalent to Cohen forcing. Hence $\QI$ below $X$ is proper
and adds a real. By \cite[Thm~4.4]{HMTU2017}, a tall analytic ideal whose quotient is
proper and adds a real is locally above $\R$; applied to $\I\upharpoonright X$ this
contradicts the Ramsey property, which is hereditary (Theorem~\ref{thm:kat}).
\end{proof}

\begin{remark}\label{rem:oscillation}
Theorem~\ref{thm:carrier} has a moral that recurs throughout the paper: the local
obstruction a negative solution must exhibit \emph{cannot} carry bounded-width or block
structure (that always contradicts $q^+$), and cannot factor through boundedly many
coordinates (that dies by \cite[Thm~3.13]{CDU2026}). It must be genuinely
oscillation-theoretic, like the $\nwd$ interval colouring --- which is exactly the kind
compatible with weak selectivity and $q^+$. See also Remark~\ref{rem:see-ideal}.
\end{remark}

\section{The barrier ladder: a saturation lemma and the Nash--Williams target}\label{sec:ladder}

Cano--Di Prisco--Uzc\'ategui \cite{CDU2026} interpolate between selectivity and the Ramsey
property by a ladder of barrier-partition properties,
\[
 \text{selective}\Rightarrow\text{semiselective}\Rightarrow\text{Galvin}
 \Rightarrow\text{Nash--Williams}\Rightarrow\text{Ramsey},
\]
where, for a barrier $\mathcal B$ on $M$, $\I$ is \emph{$\mathcal B$-Ramsey} if every finite
colouring of $\mathcal B\upharpoonright X$ ($X\in\Ipos$) has an $\Ipos$-homogeneous set;
Nash--Williams means $\mathcal B$-Ramsey for every barrier, and Galvin is the open-set
version. There is no analytic tall Galvin ideal \cite[Prop.~4.12]{CDU2026}; whether there is a
tall Borel (or analytic) Nash--Williams ideal is open, and since Nash--Williams
$\Rightarrow$ Ramsey it is a strictly weaker target than Question~\ref{q:main} (the
nondefinable neighbours of this target are the separation Questions~4.47--4.48 of
\cite{CDU2026}). Two facts from \cite{CDU2026} calibrate the region. First, the
barrier-Ramsey property transfers down in rank (\cite[Prop.~4.34]{CDU2026}), and
Nash--Williams is equivalent to $\mathcal B_\alpha$-Ramseyness for uniform barriers of
every rank $\alpha<\omega_1$ (\cite[Prop.~4.43]{CDU2026}); hence $\mathscr S$-Ramsey
already implies Ramsey, and the nonexistence targets form the chain
\begin{gather*}
 \text{no tall Borel Ramsey}\ \Rightarrow\ \text{no tall Borel }\mathscr S\text{-Ramsey}\\
 \Rightarrow\ \text{no tall Borel Nash--Williams}\ \Rightarrow\
 \text{no tall Borel Galvin (known).}
\end{gather*}
Second, for each \emph{fixed} barrier $\mathcal B$ there is a \emph{coanalytic}
$\mathcal B$-Ramsey tall ideal (\cite[Thm~4.29]{CDU2026}), so the Schreier target is
optimally sharp: it would separate Borel from coanalytic at the level of a single barrier.
We record a self-contained tool and reduce the (strongest) Schreier target to a single
colouring.

\begin{lemma}[saturation]\label{lem:saturation}
Let $\mathcal K\subseteq\I$ be hereditary and tall, and let $\mathcal B$ be a barrier. If $\I$
is $\mathcal B$-Ramsey, then $\{H:\mathcal B\upharpoonright H\subseteq\mathcal K\}$ is
positively dense: for every $X\in\Ipos$ there is $H\in\Ipos$, $H\subseteq X$, with
$\mathcal B\upharpoonright H\subseteq\mathcal K$.
\end{lemma}

\begin{proof}
Fix $X\in\Ipos$ and colour $s\in\mathcal B\upharpoonright X$ by
$c_{\mathcal K}(s)=[\,s\in\mathcal K\,]$. Since $\I$ is $\mathcal B$-Ramsey there is a
positive $H\subseteq X$ with $c_{\mathcal K}$ constant on
$\mathcal B\upharpoonright H$. By tallness of $\mathcal K$ pick an infinite
$Y\subseteq H$ with $Y\in\mathcal K$; by heredity every
$s\in\mathcal B\upharpoonright Y$ lies in $\mathcal K$, and
$\mathcal B\upharpoonright Y$ is nonempty (every infinite subset of the domain of a
barrier has initial segments in the barrier). Since
$\mathcal B\upharpoonright Y\subseteq\mathcal B\upharpoonright H$, the constant colour
on $H$ is $1$; that is, $\mathcal B\upharpoonright H\subseteq\mathcal K$.
\end{proof}

\begin{proposition}[exact form of the saturation]\label{prop:exact-sat}
Let $\mathcal K\subseteq\mathcal P(M)$ be hereditary and tall and let $\mathcal B$ be a
barrier on $M$. For the membership colouring $c_{\mathcal K}:\mathcal B\to2$,
$c_{\mathcal K}(b)=[\,b\in\mathcal K\,]$, there are no infinite homogeneous sets of colour
$0$, and
\[
 \{H\in[M]^\omega:H\text{ homogeneous for }c_{\mathcal K}\}
 =\mathrm{Sat}_{\mathcal B}(\mathcal K)
 :=\{H:\mathcal B\upharpoonright H\subseteq\mathcal K\}.
\]
In particular $\mathrm{Sat}_{\mathcal B}(\mathcal K)$ is a \emph{closed}, hereditary, tall
family, with no closure or definability hypothesis on $\mathcal K$: the saturation sees
only the finite trace $\mathcal K\cap[M]^{<\omega}$. Hence the closedness of
$\mathcal K$ is convenient but not necessary in this section.
\end{proposition}

\begin{proof}
An infinite $H$ is homogeneous for $c_{\mathcal K}$ with colour $1$ exactly when
$\mathcal B\upharpoonright H\subseteq\mathcal K$, i.e.\
$H\in\mathrm{Sat}_{\mathcal B}(\mathcal K)$. There is no infinite homogeneous $H$ of
colour $0$: tallness of $\mathcal K$ yields $Y\in\mathcal K\cap[H]^\omega$, heredity
gives $\mathcal B\upharpoonright Y\subseteq\mathcal K$, and
$\mathcal B\upharpoonright Y$ is a nonempty subset of $\mathcal B\upharpoonright H$
meeting colour $1$. This proves the displayed equality, and only membership of the
finite sets of $\mathcal B$ in $\mathcal K$ was used. Closedness is direct:
\[
 [M]^\omega\setminus\mathrm{Sat}_{\mathcal B}(\mathcal K)
 =\bigcup\bigl\{\{H\in[M]^\omega:b\subseteq H\}:b\in\mathcal B\setminus\mathcal K\bigr\}
\]
is open. Heredity holds because $L\subseteq H$ implies
$\mathcal B\upharpoonright L\subseteq\mathcal B\upharpoonright H$; and tallness holds
because, inside any infinite $H\subseteq M$, the set $Y$ above lies in
$\mathrm{Sat}_{\mathcal B}(\mathcal K)$.
\end{proof}

\begin{corollary}[obstruction criterion]\label{cor:sat-obstruction}
If there are $X\in\Ipos$, a barrier $\mathcal B$ and hereditary tall $\mathcal K\subseteq\I$
with $\{H\subseteq X:\mathcal B\upharpoonright H\subseteq\mathcal K\}\subseteq\I$, then
$\I\upharpoonright X$ is not $\mathcal B$-Ramsey.
\end{corollary}

By Proposition~\ref{prop:exact-sat}, the criterion of
Corollary~\ref{cor:sat-obstruction} is the special case $c=c_{\mathcal K}$ of the general
one: \emph{any} finite colouring $c$ of a barrier with
$\mathrm{hom}(c)\subseteq\I$ on a positive restriction witnesses failure of
$\mathcal B$-Ramseyness. We call the production of such a colouring from a hypothetical
candidate \emph{frontification}. The next obstruction shows frontification cannot be had
for free.

\begin{proposition}[selector obstruction]\label{prop:selector}
If $\mathrm{hom}(c)\subseteq\I$ for some finite colouring $c$ of a barrier, then $\I$ has a
Borel selector: the section at $c$ of the uniform Nash--Williams selector of
Greb\'ik--Uzc\'ategui (\cite{GrebikUzcategui2019}; cf.\ \cite[Lem.~4.28]{CDU2026}) is a
Borel map for \emph{every} parameter $c$, definable or not. Since there is a tall
$F_\sigma$ ideal with no Borel selector \cite{GrebikUzcategui2019}, frontification is
false for arbitrary Borel tall ideals: for a Nash--Williams candidate it must be derived
from the Ramsey-side necessary conditions (weak selectivity, $q^+$,
Theorem~\ref{thm:standard}), not from Borelness and tallness alone.
\end{proposition}

\begin{proof}
The uniform Nash--Williams theorem of \cite{GrebikUzcategui2019} provides, for a fixed
barrier $\mathcal B$, a Borel map $(A,c)\mapsto S(A,c)$ assigning to every
$A\in[\omega]^\omega$ and every finite colouring $c$ of $\mathcal B$ an infinite
$S(A,c)\subseteq A$ homogeneous for $c$ (stated there for two colours, coded as elements
of $2^{\mathcal B}$; finitely many colours follow by iterating the two-colour map, which
preserves Borelness). Fixing the second argument at our $c$ gives a
Borel map $A\mapsto S(A,c)$: the section of a Borel map at a fixed parameter is Borel,
with no definability requirement on the parameter itself. If
$\mathrm{hom}(c)\subseteq\I$, this section picks inside every infinite $A$ an infinite
subset lying in $\I$ --- a Borel selector for $\I$. In particular, if
$\mathrm{Sat}_{\mathcal B}(\mathcal K)\subseteq\I$ for some hereditary tall
$\mathcal K$, the case $c=c_{\mathcal K}$ applies by
Proposition~\ref{prop:exact-sat}. Since \cite{GrebikUzcategui2019} construct a tall
$F_\sigma$ ideal with no Borel selector, that ideal contains no family of the form
$\mathrm{hom}(c)$; the remaining assertions follow.
\end{proof}

For the Schreier barrier $\mathscr S=\{s:|s|=\min(s)+1\}$, a single \emph{closed} hereditary
tall $\mathcal K\subseteq\I$ yields the coherent, increasing-dimension colouring
\[
 c^{\mathcal K}_n:[\omega\setminus(n+1)]^n\to2,\qquad
 c^{\mathcal K}_n(t)=[\,\{n\}\cup t\in\mathcal K\,],
\]
whose sections come from one closed object and are coherent by heredity. Unlike a min-,
membership-, or interval-type colouring, $c^{\mathcal K}$ \emph{sees} $\I$ through
$\mathcal K$ (Remark~\ref{rem:see-ideal}). By Corollary~\ref{cor:sat-obstruction}, ``no tall
Borel Nash--Williams ideal'' would follow from an affirmative answer to
Question~\ref{q:schreier} of \S\ref{sec:targets}.

\begin{lemma}[coordination across ranks]\label{lem:coord}
Let $\I$ be $\mathcal B$-Ramsey for a uniform barrier $\mathcal B$, let
$\mathcal C_1,\dots,\mathcal C_m$ be barriers with
$\mathrm{rk}(\mathcal C_i)<\mathrm{rk}(\mathcal B)$, and let $\mathcal K\subseteq\I$ be
hereditary and tall. For every $M\in\Ipos$ there is $H\in\Ipos$, $H\subseteq M$, with
$\mathcal B\upharpoonright H\subseteq\mathcal K$ and, for each $i\le m$,
$\mathcal C_i\upharpoonright H\sqsubseteq\mathcal B\upharpoonright H$ and
$\mathcal C_i\upharpoonright H\subseteq\mathcal K$. In particular a Nash--Williams
candidate saturates $\mathcal K$ coherently on every \emph{finite} window of ranks: for
any $n_1<\dots<n_m$ there is one positive $H$ with $[H]^{n_i}\subseteq\mathcal K$ for all
$i\le m$.
\end{lemma}

\begin{proof}
Colour $b\in\mathcal B\upharpoonright M$ by the vector
\[
 c(b)=\bigl([\,b\in\mathcal K\,],\
 [\,\exists s\in\mathcal C_1\ s\sqsubseteq b\,],\ \dots,\
 [\,\exists s\in\mathcal C_m\ s\sqsubseteq b\,]\bigr)\in2^{m+1}.
\]
The $\mathcal B$-Ramsey property holds for any finite number of colours (see the remarks
following \cite[Def.~4.25]{CDU2026}), so there is $H\in\Ipos$, $H\subseteq M$,
homogeneous for $c$. In the first coordinate the colour $0$ is impossible on any
infinite set by Proposition~\ref{prop:exact-sat}; hence
$\mathcal B\upharpoonright H\subseteq\mathcal K$. In coordinate $i+1$ the colour $0$ is
impossible by the rank-comparison argument of \cite[Lem.~4.33]{CDU2026} applied to
$\mathcal B\upharpoonright H$ and $\mathcal C_i\upharpoonright H$ (the restriction of a
uniform barrier to an infinite subset of its domain is uniform of the same rank). So the
colour of $H$ is $(1,1,\dots,1)$, which is
$\mathcal B\upharpoonright H\subseteq\mathcal K$ together with
$\mathcal C_i\upharpoonright H\sqsubseteq\mathcal B\upharpoonright H$ for each $i\le m$.
For the containments $\mathcal C_i\upharpoonright H\subseteq\mathcal K$: given
$s\in\mathcal C_i\upharpoonright H$, by \cite[Fact~4.31(2)]{CDU2026} there is
$b\in\mathcal B\upharpoonright H$ with $s\sqsubseteq b$; then
$s\subseteq b\in\mathcal K$ and heredity of $\mathcal K$ gives $s\in\mathcal K$.
The last sentence of the statement is the case $\mathcal C_i=[\omega]^{n_i}$ inside a
uniform barrier $\mathcal B$ of rank above $n_m$ (for a Nash--Williams candidate, e.g.\
$\mathcal B=[\omega]^{m'}$ with $m'>n_m$).
\end{proof}

\begin{remark}[the finite-to-$\omega$ jump is the same wall]\label{rem:window}
For closed $\mathcal K$ (e.g.\ the Greb\'ik--Vidny\'anszky--Mazur witness) the
$\omega$-window is outright contradictory: one positive $H$ with
$[H]^n\subseteq\mathcal K$ for \emph{all} $n$ would give $H\in\mathcal K\subseteq\I$. So a
Nash--Williams candidate coherently saturates every finite window of ranks while provably
never saturating the $\omega$-window; the missing step is a fusion/properness step, i.e.\
precisely the local-properness principle of \S\ref{sec:core}. The same pattern defeats
every barrier colouring assembled from $\sigma$-small data (tower exit functions,
partitions into small pieces, submeasure masses): the candidate's necessary conditions
manufacture a positive homogeneous escape --- either a forbidden positive
transversal/pseudointersection, forcing the data to have a positive piece, or a
homogeneous set inside that piece, where the colouring degenerates.
\end{remark}

\begin{proposition}[Nash--Williams forces unbounded-rank positive barriers inside the small
witness]\label{prop:nw-unbounded}
If $\I$ is Nash--Williams and $\mathcal K\subseteq\I$ is hereditary tall, then
\[
 \mathfrak B^+_{\I,\mathcal K}=\{\mathcal C:\mathcal C\text{ a barrier on }
 \textstyle\bigcup\mathcal C\in\Ipos,\ \mathcal C\subseteq\mathcal K\}
\]
is unbounded in rank.
\end{proposition}

\begin{proof}
Let $2\le\alpha<\omega_1$ and fix a uniform barrier $\mathcal B_\alpha$ of rank $\alpha$
on $\omega$. Since $\I$ is Nash--Williams, it is $\mathcal B_\alpha$-Ramsey, and
Lemma~\ref{lem:saturation} gives $H_\alpha\in\Ipos$ with
$\mathcal C_\alpha:=\mathcal B_\alpha\upharpoonright H_\alpha\subseteq\mathcal K$. The
restriction of a uniform barrier to an infinite subset of its domain is a uniform
barrier of the same rank, so $\mathrm{rk}(\mathcal C_\alpha)=\alpha$; and every point of
$H_\alpha$ lies in some member of $\mathcal C_\alpha$, so
$\bigcup\mathcal C_\alpha=H_\alpha\in\Ipos$. Hence
$\mathcal C_\alpha\in\mathfrak B^+_{\I,\mathcal K}$ for every $\alpha$, and the family
is unbounded in rank.
\end{proof}

\begin{remark}[why the boundedness route is delicate]\label{rem:boundedness-wall}
The family $\mathfrak B^+_{\I,\mathcal K}$ is coanalytic, since ``$\mathcal C$ is a barrier''
is $\Pi^1_1$-complete \cite[Prop.~4.45]{CDU2026}, independently of $\I$; and $\mathrm{rk}$ is a
coanalytic rank on fronts \cite[Prop.~4.46]{CDU2026}. By $\Sigma^1_1$-boundedness, every
\emph{analytic} family of fronts is bounded in rank; hence
Proposition~\ref{prop:nw-unbounded} admits \emph{no} unbounded analytic subfamily, and the
only route through boundedness is to prove $\mathfrak B^+_{\I,\mathcal K}$ \emph{itself}
analytic --- a $\Pi^1_1\to\Sigma^1_1$ collapse of the barrier predicate for which the
Borelness of $\I$ supplies no evident mechanism. This is the same coanalytic-versus-Borel
wall as in Sections~\ref{sec:core} and~\ref{sec:targets}. Consequently the viable outcome
of this section is the single-colouring Question~\ref{q:schreier}, which avoids ranks and
$\omega_1$ altogether.
\end{remark}

\subsection{Obstructions must see the ideal}\label{ssec:see-ideal}

\begin{proposition}[iterated-Ramsey diagonal]\label{prop:diagonal}
Let $c_k:[\omega]^{n_k}\to r_k$ \textup($k\in\omega$, $n_k,r_k<\omega$\textup) be finite
colourings of finite dimensions. Every infinite $C\subseteq\omega$ has an infinite
$D\subseteq C$ almost homogeneous for every $c_k$ simultaneously \textup(homogeneous
after a finite cut\textup). If $C$ is a sequence converging in a metric space, then so is
$D$, with the same limit.
\end{proposition}

\begin{proof}
Recursively put $A_0=C$ and let $A_{k+1}\subseteq A_k$ be infinite and
$c_k$-homogeneous, by Ramsey's theorem. Choose $d_0<d_1<\cdots$ with $d_k\in A_k$ and set
$D=\{d_k:k\in\omega\}$. For each $k$, $D\setminus\{d_0,\dots,d_k\}\subseteq A_{k+1}$, so
$c_k$ is constant on $[D\setminus\{d_0,\dots,d_k\}]^{n_k}$. A subsequence of a convergent
sequence converges to the same limit.
\end{proof}

\begin{remark}\label{rem:see-ideal}
A colouring assembled from a fixed countable amount of ideal-free data (an enumeration, a
fixed geometric function) cannot witness failure of the Ramsey property: every
finite-dimensional colouring is $\sigma$-decided \cite[Thm~3.13]{CDU2026}, and every
countable ideal-free family is defeated by the iterated-Ramsey diagonal of
Proposition~\ref{prop:diagonal} --- inside every infinite set, in particular inside every
positive set, there are infinite sets almost homogeneous for the whole family, and
nothing built without reference to $\I$ forces these diagonal sets into $\I$. Any genuine
obstruction must depend on $\I$ itself. The saturation colouring $c^{\mathcal K}$ does so
through $\mathcal K\subseteq\I$; this is precisely why it escapes the filter that kills
min-, membership-, and interval-type colourings.
\end{remark}

\begin{remark}[canonical rigidity of interval separation is ideal-free]\label{rem:route3}
As an illustration of Remark~\ref{rem:see-ideal}: in the reduced presentation, for a positive
convergent sequence $c_n\downarrow z$ the interval-separation function
$f(\{x,y\})=\min\{n:q_n\in(x,y)\}$ has, on a positive subset, an Erd\H os--Rado canonical form
\cite[Thm~3.30]{CDU2026} that is min-determined or injective, \emph{never} constant or
max-determined (the constant and max-determined types force, respectively, disjoint intervals
to share a rational and a strictly decreasing $\omega$-sequence). This rigidity holds for any
monotone convergent sequence in $\Q$ with no reference to $\I$, so by
Remark~\ref{rem:see-ideal} it cannot obstruct the Ramsey property.
\end{remark}

\section{Width and continuous reading of names}\label{sec:core}

This section proves face~(c) of Theorem~A and isolates the reduction principles (DW) and
(RC). Throughout, $\I$ is a tall Borel ideal with the Ramsey
property unless stated otherwise, $a$ is a fixed Borel code for $\I$, and $\QI$ is the
quotient forcing of \S\ref{sec:prelim}. The situation is the residual case left open by
\cite[Thm~3.18]{HMTU2017}: by (F3) below, $\QI$ adds a real and is not proper, while by
\cite[Thm~3.13]{HMTU2017} it is $p^+_{\mathrm{tower}}$ and $(\omega,2)$-distributive ---
it decides ground-model countable families modulo $\I$, which is weaker than adding no
real --- yet not $\sigma$-distributive. In the proof of \cite[Lem.~3.17]{HMTU2017},
properness is used exactly once, to make the maximal antichains deciding a new real
countable on a positive set. This section deletes that use and isolates what remains. We
use the following toolkit from \cite{HMTU2017}.

\begin{lemma}[toolkit; \cite{HMTU2017}]\label{lem:toolkit}
Let $\I$ be a tall analytic ideal.
\begin{itemize}
\item[\textup{(F1)}] $\R\leK\conv$, and $\conv\leK\I\upharpoonright X$ for some $X\in\Ipos$
implies $\R\leK\I\upharpoonright X$, hence $\I\upharpoonright X$ is not Ramsey
(\cite[Lem.~4.3, Cor.~4.2]{HMTU2017}).
\item[\textup{(F2)}] (\emph{Coherence}, \cite[Lem.~3.16]{HMTU2017}.) If $(\mathcal P_n)_{n}$
is a refining $\omega$-sequence of $\I\upharpoonright X$-decompositions (partitions of $X$
into positive pieces detecting $\I$ coordinatewise) and every pseudo-intersection of every
decreasing branch $P_0\supseteq P_1\supseteq\cdots$ ($P_n\in\mathcal P_n$) lies in $\I$, then
$\conv\leK\I\upharpoonright X$.
\item[\textup{(F3)}] (\cite[Thm~4.4]{HMTU2017}.) If the quotient of $\I$ adds no real, or
is proper and adds a real, then $\I$ is locally above $\R$. Hence a tall Borel \emph{Ramsey}
$\I$ has a quotient $\QI$ that \emph{adds a real} and is \emph{not proper}.
\item[\textup{(F4)}] (\cite[Lem.~3.17]{HMTU2017}, proof.) If $\QI$ is proper and adds a real
then $\conv\leK\I\upharpoonright X$ for some $X$. The proof fixes a name $\dot r$ for a new
real and refining maximal antichains $(\mathcal A_n)_n$ with $A\Vdash\dot r\upharpoonright
n=\sigma_A$ for $A\in\mathcal A_n$; \emph{properness is used only once}, to produce
$X\in\Ipos$ making each $B_n(X)=\{A\in\mathcal A_n: A\cap X\in\Ipos\}$ countable; the
$B_n(X)$ are then refined to decompositions and \textup{(F2)} applies.
\end{itemize}
\end{lemma}

\subsection{Countable local width of the deciding antichains}

Removing properness from (F4) isolates the exact residue; the theorem below is
\cite[Lem.~3.17]{HMTU2017} with properness deleted.

\begin{definition}[local width]\label{def:width}
For a maximal antichain $\mathcal A$ of $\QI$ and $X\in\Ipos$ put
$B_{\mathcal A}(X)=\{A\in\mathcal A:A\cap X\in\Ipos\}$, the set of members of $\mathcal A$
that stay positive below $X$. Say a name $\dot r$ for a real has \emph{countable width below
$X$} if its refining deciding antichains $(\mathcal A_n)_n$ satisfy
$|B_{\mathcal A_n}(X)|\le\aleph_0$ for every $n$.
\end{definition}

\begin{theorem}[reduction]\label{thm:reduction}
Let $\I$ be tall Borel with the Ramsey property (so, by \textup{(F3)}, $\QI$ adds a real and
is not proper). If \emph{some} name for a new real has countable width below \emph{some}
$X\in\Ipos$, then $\conv\leK\I\upharpoonright X$, contradicting Ramsey. Equivalently: for a
tall Borel Ramsey $\I$, \emph{every} name for a new real has uncountable width below
\emph{every} positive set.
\end{theorem}

\begin{proof}
Let $\dot r$ be a name for a new real with refining deciding antichains $(\mathcal A_n)$ and
suppose $B_n:=B_{\mathcal A_n}(X)$ is countable for all $n$. Enumerate $B_n=\{A^n_k:k\}$ and,
using that $\mathcal A_n$ is a maximal antichain, refine below $X$ to a decomposition
$\mathcal P_n$ of $X$ whose pieces each lie below a single $A^n_k$ (split $X$ along the
countable partition induced by $B_n$; the members of $\mathcal A_n$ outside $B_n$ meet $X$ in
$\I$-small sets and are absorbed). Arrange $\mathcal P_{n+1}$ to refine $\mathcal P_n$ (the
$\mathcal A_n$ already refine; take common refinements below $X$, still countable). Each
partition so produced is a decomposition in the sense of \textup{(F2)}: a positive
$A\subseteq X$ is compatible with some member of the maximal antichain $\mathcal A_n$ ---
necessarily one lying in $B_n$ --- and hence meets the corresponding piece positively; for
a common refinement, apply this at both levels. Now let
$P_0\supseteq P_1\supseteq\cdots$, $P_n\in\mathcal P_n$, be a decreasing branch and $Y$ a
pseudo-intersection. Each $P_n$ lies below some $A^n_{k(n)}\in\mathcal A_n$, so
$Y\le_\I A^n_{k(n)}$ for all $n$, whence $Y\Vdash\dot r\upharpoonright n=\sigma_{A^n_{k(n)}}$
for all $n$; thus $Y$ forces $\dot r=\bigcup_n\sigma_{A^n_{k(n)}}\in V$. As $\dot r$ is
forced to be new, no positive condition forces $\dot r\in V$, so $Y\in\I$. Every branch
pseudo-intersection is $\I$-small; by \textup{(F2)}, $\conv\leK\I\upharpoonright X$, and by
\textup{(F1)} $\I\upharpoonright X$ is not Ramsey. The equivalent phrasing is the
contrapositive.
\end{proof}

\begin{corollary}[the sufficient boundedness principle]\label{cor:DW}
The negative answer to Question~\ref{q:main} follows from:
\begin{quote}\itshape
\textup{(DW)} Every tall Borel ideal whose quotient adds a real admits a new-real name with
countable width below some positive set.
\end{quote}
\end{corollary}

\begin{remark}[what \textup{(DW)} really asks, and why it is hard]\label{rem:DW-hard}
Theorem~\ref{thm:reduction} turns ``$\QI$ proper'' into the strictly local ``some new-real
name is countable-width below some $X$''. Three cautions.
\begin{enumerate}
\item \emph{Ramsey forbids the conclusion.} For a Ramsey candidate, (F3) gives non-properness,
and by Theorem~\ref{thm:reduction} \emph{every} name for a new real has uncountable width
below \emph{every} positive set. So (DW) cannot be verified ``by hand'' on a candidate: it must be
derived, in ZFC, from Borelness of the code $a$, in a form that the Ramsey hypothesis cannot
block --- a genuine boundedness\slash reflection theorem, not a construction.
\item \emph{The coanalytic barrier.} The ZFC coanalytic tall Ramsey ideal
(\cite[Thm~4.8]{HMTU2017}) is a definable non-proper forcing for which (DW) fails. Hence (DW)
must separate Borel from coanalytic: it is false for a $\Pi^1_1$ code and would have to be
true for a $\Delta^1_1$ one.
\item \emph{The missing input is a definable name.} If the new real $\dot r$ and the antichains
$(\mathcal A_n)$ could be chosen \emph{analytic in $a$}, then ``$A\cap X\in\Ipos$ and
$A\in\mathcal A_n$'' would be an analytic relation with the width $B_n(X)$ as its sections, and
a Luzin--Novikov / $\Sigma^1_1$-boundedness argument would be available to seek a positive $X$
with countable sections. There is no evident mechanism producing such a definable name for the
real added by a Borel non-proper quotient; this is exactly the open gap.
\end{enumerate}
\end{remark}

\subsection{The twin: reflection of new colourings is continuous reading of names}

Recall (\cite[Thm~4.4]{HMTU2017}, proof) that for a tall Ramsey $\I$ the generic ultrafilter
$U$ (a $V$-generic filter on $\QI$, disjoint from $\I$) is selective with respect to $V$; it
cannot be selective in $V[G]$, for a selective ultrafilter disjoint from a tall ideal
contradicts Mathias. Thus $V[G]$ contains a \emph{new} colouring $\dot c:[\omega]^2\to2$ (or
partition) with \emph{no $U$-homogeneous set}.

\begin{proposition}[the witnessing colouring does not reflect]\label{prop:noreflect}
Let $\I$ be tall Borel Ramsey and let $\dot c$ be a new colouring with no $U$-homogeneous set,
as above. Then there is no positive $A\in U$ and no ground-model colouring $c\in V$ with
$A\Vdash\dot c\upharpoonright[A]^2=c\upharpoonright[A]^2$.
\end{proposition}

\begin{proof}
Suppose such $A\in U$ and $c\in V$ existed. Since $\I$ has the ordinary Ramsey property, the
set of conditions containing a $c$-homogeneous positive subset is dense below $A$; as $U$ is
generic, some $c$-homogeneous $H\in U$ with $H\subseteq A$ exists. Then $H\Vdash\dot
c\upharpoonright[H]^2=c\upharpoonright[H]^2$ is constant, so $H$ is $\dot c$-homogeneous and
$H\in U$, contradicting the choice of $\dot c$.
\end{proof}

\begin{corollary}[reflection $\Rightarrow$ not Ramsey]\label{cor:RC}
The negative answer to Question~\ref{q:main} follows from:
\begin{quote}\itshape
\textup{(RC)} For a tall Borel $\I$ and every $\QI$-name $\dot c$ for a colouring
$[\omega]^2\to2$, the conditions $A$ with $A\Vdash\dot c\upharpoonright[A]^2\in V$ are dense.
\end{quote}
\end{corollary}

\begin{proof}
If (RC) holds, $U$ meets that dense set: some $A\in U$ reads $\dot c$ continuously into a
ground-model $c$. Then $A,c$ contradict Proposition~\ref{prop:noreflect}. So no such $\dot c$
(hence no witness to non-selectivity of $U$ in $V[G]$) exists; $U$ is selective in $V[G]$,
contradicting Ramsey via Mathias as in \cite[Thm~4.4]{HMTU2017}.
\end{proof}

\begin{remark}[unification: both faces are ``Borelness $\Rightarrow$ local properness'']
\label{rem:unify}
Property (RC) is precisely a \emph{continuous reading of names} for colourings on a positive
condition. Continuous reading of names is the defining regularity of definable \emph{proper}
forcings (Zapletal's idealised forcing); (F4) is the same phenomenon for names of reals. Thus
Corollary~\ref{cor:DW} and Corollary~\ref{cor:RC} are two forms of one statement:
\begin{quote}\itshape
Borelness of the code forces a properness-like reading of $\QI$-names on \emph{some} positive
condition, even though $\QI$ is globally non-proper.
\end{quote}
Proposition~\ref{prop:noreflect} shows the reading genuinely fails for the specific new object
a Ramsey candidate must add, so neither (DW) nor (RC) is a soft consequence of the necessary
conditions; each is a $\Pi^1_1\!\to\!\Sigma^1_1$ (coanalytic-to-analytic) collapse of the
name structure, unavailable for the coanalytic example and with no known Borel mechanism.
\end{remark}

\section{The main theorem and the converse question}\label{sec:targets}

This final section assembles the results of the paper into a single statement
(Theorem~\ref{thm:master}), records a forcing-absoluteness transfer available to any
future attempt on Question~\ref{q:main}, summarizes the constraints established, and
states the three questions that remain.

\subsection{No countable local reading}\label{ssec:master}

The following theorem is a synthesis: its proof consists of assembling results proved in
Sections~\ref{sec:countable}--\ref{sec:core}. We use \emph{reading} as an umbrella term:
faces (a)--(c) concern literal countable readings of positivity --- by represented ideals,
by dense sets of conditions, by deciding antichains --- while face (d) is the analogous
failure of countable saturation in the barrier dimension.

\begin{theorem}[no countable local reading]\label{thm:master}
Let $\I$ be a tall Borel Ramsey ideal and let $X\in\Ipos$. Then:
\begin{enumerate}
\item[\textup{(a)}] \textup{(representations)} $\I\upharpoonright X$ is not the
intersection of countably many ideals each failing the ordinary Ramsey property --- in
particular, of countably many topologically represented ideals --- nor of countably many
tall analytic $P$-ideals; in particular the Kwela--Sabok
representation of $\I\upharpoonright X$ has no countable subfamily with the same
intersection.
\item[\textup{(b)}] \textup{(quotient topology)} $\mathcal P(\omega)/\I$ has no countable
dense subset below $X$.
\item[\textup{(c)}] \textup{(forcing names)} every $\QI$-name for a new real has
uncountable width below $X$, and the colouring witnessing non-selectivity of the generic
ultrafilter is not read continuously into the ground model on any positive
$A\subseteq X$.
\item[\textup{(d)}] \textup{(barriers)} for every closed hereditary tall
$\mathcal K\subseteq\I$ and all dimensions $n_1<\dots<n_m$ there is one positive
$H\subseteq X$ with $[H]^{n_i}\subseteq\mathcal K$ for every $i\le m$, yet no positive
$H\subseteq X$ satisfies $[H]^n\subseteq\mathcal K$ for all $n$ simultaneously.
\end{enumerate}
\end{theorem}

\begin{proof}
Since $\I\upharpoonright X$ is again a tall Borel Ramsey ideal (Theorem~\ref{thm:kat}),
all four parts localize, and we may quote the corresponding global statements. Part~(a) is
Theorems~\ref{thm:int-sub} and~\ref{thm:int-nonramsey} applied to $\I\upharpoonright X$,
together with Corollary~\ref{cor:frontier}. Part~(b) is Theorem~\ref{thm:qdense}. Part~(c)
is Theorem~\ref{thm:reduction} together with Proposition~\ref{prop:noreflect}. For
part~(d): a Ramsey ideal is $[\omega]^n$-Ramsey for every $n$ \cite[Thm~3.13]{CDU2026}, so
the coherent saturation of the finite window $n_1<\dots<n_m$ is Lemma~\ref{lem:coord}
applied to the uniform barriers $[\omega]^{n_i}$ inside $[\omega]^{m'}$, $m'>n_m$; and if
some positive $H$ had $[H]^n\subseteq\mathcal K$ for every $n$, then every finite subset
of $H$ would lie in the hereditary family $\mathcal K$, whence $H\in\mathcal K$ by
closedness --- contradicting $H\notin\I\supseteq\mathcal K$.
\end{proof}

Informally: no countable amount of definable data --- representing ideals, dense sets of
conditions, deciding antichains, barrier dimensions --- reads the positivity of a Borel
Ramsey ideal below any positive set. This is exactly the freedom the coanalytic example
enjoys, there sustained by the unbounded ranks of $\Pi^1_1$ sets; the converse question of
\S\ref{ssec:central} asks whether a Borel code can afford it.

\subsection{A forcing-absoluteness freedom}\label{sec:absolute}

The following is a routine instance of Shoenfield absoluteness; we record it, claiming no
novelty, because of the freedom it gives to any proof of a negative answer.

\begin{proposition}[Shoenfield transfer]\label{prop:absolute}
For a fixed Borel code, ``$\I$ is a tall Ramsey ideal'' is $\Pi^1_2$ and ``$\R\leK
\I\upharpoonright X$ for some $X\in\Ipos$'' is $\Sigma^1_2$; both are absolute (Shoenfield).
Hence if ZFC${}+{}$CH (or ${}+{}$MA, ${}+{}\mathrm{cov}(\mathcal M)=\mathfrak c$, \dots)
proves Question~\ref{q:main} has a negative answer, then ZFC does.
\end{proposition}

\begin{proof}
Fix the code. Being a proper ideal containing the finite sets is $\Pi^1_1$ in the code
(closure under subsets and unions quantifies universally over reals with a Borel matrix);
tallness is $(\forall X\in[\omega]^\omega)(\exists Y\in[X]^\omega)\,Y\in\I$, and the
Ramsey property is
$(\forall X)(\forall c)(\exists H)\,[\,X\in\Ipos\wedge c:[X]^2\to2\ \rightarrow\
H\in\Ipos\wedge H\subseteq X\wedge c\text{ constant on }[H]^2\,]$; in both, the matrix is
Borel in the code, so the conjunction is $\Pi^1_2$. Likewise
``$(\exists X)(\exists f)\,[\,X\in\Ipos$ and $f$ witnesses
$\R\leK\I\upharpoonright X\,]$'' is $\Sigma^1_2$: the inner clause
$(\forall A)\,[\,A\in\R\rightarrow f^{-1}[A]\cap X\in\I\,]$ is $\Pi^1_1$ because $\R$ and
$\I$ are Borel. Shoenfield's absoluteness applies to both. For the transfer, suppose
ZFC${}+{}$CH proved the negative answer, and let $a$ be a Borel code of a putative tall
Ramsey ideal in $V$. Force CH; in $V[G]$ the code $a$ still defines a tall Borel Ramsey
ideal ($\Pi^1_2$-absoluteness), contradicting the assumed theorem in $V[G]$. The
argument is uniform in $a$, so ZFC alone proves the negative answer. The other listed
hypotheses are forced the same way (MA by the usual finite-support ccc iteration, etc.).
\end{proof}

\begin{remark}[transfer template]\label{rem:transfer}
The operative form of Proposition~\ref{prop:absolute} is downward transfer of a \emph{new}
local witness: if in some forcing extension $V[G]$ one produces reals $X,f$ with
$\R\leK\I\upharpoonright X$ (the $\Sigma^1_2$ statement, possibly using generic reals), then
such a witness already exists in $V$. So a forced hypothesis (CH, MA) may be used to
\emph{fabricate} the obstruction, not merely to assume an axiom. The argument in $V[G]$ must
end in the $\Sigma^1_2$ conclusion itself, not in ``the generic destroys the ground-model
objects'' --- the latter does not quantify over the new objects. The global statement
``there is a Borel tall Ramsey ideal'' is $\Sigma^1_3$ (it begins $\exists$ code), so this is
used per fixed counterexample code, never on the global sentence.
\end{remark}

\begin{remark}[why forcing cannot refute: repopulation]\label{rem:absolute-caveat}
Absoluteness also \emph{protects} the coanalytic example: in any extension $\I$ stays Ramsey,
so a generic (Cohen/Rado) colouring still has a positive homogeneous set. The mechanism is
concrete. Any family $\mathcal A$ defined from the Borel code (nodes, residuals, positive
sets) is \emph{repopulated} in $V[G]$: $\mathcal A^{V[G]}\supsetneq\mathcal A^V$ and, when
$\mathfrak c^{V[G]}>\omega_1$, $|\mathcal A^{V[G]}|=\mathfrak c$. A splitter of the old
$\mathcal A^V$ need not split the new members, and the old family need not remain cofinal for
the new positives; this is exactly how the coanalytic example survives splitter-adding
extensions. Consequently the amplifier ``$\mathrm{MA}\Rightarrow\mathfrak r=\mathfrak c$
splits every family of size $<\mathfrak c$'' bites only against a family that is cofinal for
\emph{all} reals of the extension, not just the ground-model ones. The still-missing input
is therefore an $\omega_1$-reflection: a size-$\omega_1$ family, arising definably from
the ideal in every universe, cofinal for all current positives and unsplittable. Even
then, a splitter of such a family refutes nothing by itself --- it must be converted,
Borel-uniformly and absolutely, into a countable-width name for a new real
(Theorem~\ref{thm:reduction}) or a continuous reading of a new colouring
(Corollary~\ref{cor:RC}); the forcing route thus terminates in Questions~\ref{q:DW}
and~\ref{q:RC} below.
\end{remark}

The formats in which such a family could be sought line up as follows; both parts of the
lemma are immediate from classical facts, and we record them claiming no novelty.

\begin{lemma}[formats]\label{lem:format}
Let $\mathcal N\subseteq[\omega]^\omega$.
\begin{enumerate}
\item If $\mathcal N$ is coanalytic with a bounded $\Pi^1_1$-rank, then $\mathcal N$ is
Borel, hence analytic.
\item Assume \textup{CH}. If $\mathcal N$ is analytic, then $|\mathcal N|\le\omega_1$.
\end{enumerate}
\end{lemma}

\begin{proof}
(1) A coanalytic set is Borel if and only if some (equivalently, every) $\Pi^1_1$-rank on
it is bounded, by the boundedness theorem for $\Pi^1_1$-ranks. (2) An uncountable analytic
set contains a perfect set, so under \textup{CH} every analytic set has cardinality at
most $\omega_1$.
\end{proof}

\begin{remark}[the formats are not equivalent]\label{rem:format}
The converses of Lemma~\ref{lem:format} are not claimed, and the size-$\omega_1$ format is
model-sensitive: under $\mathrm{MA}+\mathfrak c>\omega_1$ there is no unsplittable
(reaping) family of size $\le\omega_1$, so a size-$\omega_1$ family of the kind described
in Remark~\ref{rem:absolute-caveat} can exist only in models of $\mathfrak r\le\omega_1$,
such as models of \textup{CH} --- which, by Proposition~\ref{prop:absolute}, may always be
assumed when the goal is a negative answer to Question~\ref{q:main}.
\end{remark}

\subsection{Summary of constraints}\label{ssec:nogo}

The following table collects the forms of argument on Question~\ref{q:main} that the
results of this paper exclude, together with the statements excluding them.

\begin{center}\small
\begin{tabular}{@{}p{0.53\textwidth}p{0.39\textwidth}@{}}
\hline
\emph{Excluded form of argument} & \emph{Excluding statement}\\
\hline
Local obstructions with bounded-width or block structure, or factoring through boundedly
many coordinates & Theorem~\ref{thm:carrier}, Remark~\ref{rem:oscillation};
\cite[Thm~3.13]{CDU2026}\\[3pt]
Colourings built from countable ideal-free data, including canonical forms of interval
separation & Remarks~\ref{rem:see-ideal} and~\ref{rem:route3}\\[3pt]
Carrier programs: $\EDfin$ above a subideal on a positive set &
Theorem~\ref{thm:carrier}\\[3pt]
Countable intersections of represented ideals; countable subfamilies of the
representation & Theorems~\ref{thm:int-sub} and~\ref{thm:int-nonramsey};
Corollary~\ref{cor:frontier}\\[3pt]
Unbounded analytic families of barriers (for Nash--Williams candidates) &
Proposition~\ref{prop:nw-unbounded}, Remark~\ref{rem:boundedness-wall}\\[3pt]
Countable positive $\pi$-bases mod finite & Theorem~\ref{thm:width}\\[3pt]
Low-complexity reals added by the quotient & Propositions~\ref{prop:Brep}
and~\ref{prop:geom} (Appendix~\ref{app:completion})\\[3pt]
Refuting the candidate by destroying ground-model objects in a forcing extension &
Remarks~\ref{rem:transfer} and~\ref{rem:absolute-caveat}\\
\hline
\end{tabular}
\end{center}

\subsection{The converse question}\label{ssec:central}

Theorem~\ref{thm:master} says that a Borel Ramsey ideal escapes every countable local
reading, and the coanalytic example shows that $\Pi^1_1$ codes genuinely have this
freedom. The remaining content of Question~\ref{q:main} is the converse: \emph{does a
Borel code force a countable, properness-like reading somewhere?} The two halves of
face~(c) give this converse a precise shape, and each of the following three forms has its
consequence proved in this paper.

\begin{question}[countable width]\label{q:DW}
Does every tall Borel ideal whose quotient adds a real admit a name for a new real with
countable width below some positive set --- that is, does \textup{(DW)} hold?
\end{question}

An affirmative answer implies a negative answer to Question~\ref{q:main}
(Corollary~\ref{cor:DW}). Remark~\ref{rem:DW-hard} records why (DW) would have to be a
boundedness\slash reflection theorem rather than a construction, and why it fails for
$\Pi^1_1$ codes.

\begin{question}[continuous reading]\label{q:RC}
For every tall Borel ideal $\I$ and every $\QI$-name $\dot c$ for a colouring of pairs,
are the conditions reading $\dot c$ into the ground model dense --- that is, does
\textup{(RC)} hold?
\end{question}

An affirmative answer implies a negative answer to Question~\ref{q:main}
(Corollary~\ref{cor:RC}). By Remark~\ref{rem:unify}, (DW) and (RC) are two faces of a
single principle: Borelness of the code should force a properness-like reading of
$\QI$-names on some positive condition, even though $\QI$ is globally non-proper.

The third form is a \emph{frontification principle} in the sense of \S\ref{sec:ladder}:
although phrased as an implication about $\mathscr S$-Ramsey candidates, its content is
the uniform production of the witness $\mathcal K$ from the candidate's necessary
conditions, not the ensuing logical implication.

\begin{question}[a frontification principle for the Schreier barrier]\label{q:schreier}
Does every tall Borel $\mathscr S$-Ramsey ideal $\I$ admit $X\in\Ipos$ and a closed
hereditary tall $\mathcal K\subseteq\I$ with
$\{H\subseteq X:\mathscr S\upharpoonright H\subseteq\mathcal K\}\subseteq\I$?
\end{question}

\begin{proposition}\label{prop:schreier-suff}
An affirmative answer to Question~\ref{q:schreier} implies that there is no tall Borel
$\mathscr S$-Ramsey ideal, and hence no tall Borel Nash--Williams ideal.
\end{proposition}

\begin{proof}
Suppose $\I$ is tall Borel $\mathscr S$-Ramsey and let $X,\mathcal K$ be as provided. By
Corollary~\ref{cor:sat-obstruction}, $\I\upharpoonright X$ is not $\mathscr S$-Ramsey;
since the $\mathcal B$-Ramsey properties pass to positive restrictions, this contradicts
$\mathscr S$-Ramseyness of $\I$. The second implication is immediate, as Nash--Williams
means $\mathcal B$-Ramsey for every barrier.
\end{proof}

Question~\ref{q:schreier} is the weakest of the three and the sharpest in complexity: by
\cite[Prop.~4.34, Prop.~4.43]{CDU2026}, $\mathscr S$-Ramseyness already implies the Ramsey
property, so the target ``no tall Borel $\mathscr S$-Ramsey ideal'' is intermediate
between Question~\ref{q:main} and the Nash--Williams version; and for each fixed barrier
there is a coanalytic $\mathcal B$-Ramsey tall ideal (\cite[Thm~4.29]{CDU2026}), so an
affirmative answer would separate Borel from coanalytic at the level of a single barrier.
Proposition~\ref{prop:selector} bounds the possible proofs: the required $\mathcal K$
cannot be produced from Borelness and tallness alone, and the mass choice
$\mathcal K=\{A:\varphi(A)\le1\}$ provably fails, so $\mathcal K$ must come from the
below-$\nwd$ structure of the candidate (Section~\ref{sec:reduced}). Whether the
coanalytic example is $\mathscr S$-Ramsey (\cite[Q.4.39]{CDU2026}, itself open) serves as
calibration. The following reformulation makes the question concrete.

\begin{proposition}[prescribed identity cuts]\label{prop:schreier-cuts}
$\I$ is $\mathscr S$-Ramsey iff for every positive $X$ and every sequence
$c_n:[X\setminus(n+1)]^n\to2$ there is a positive $H\subseteq X$ with each $c_n$ constant
on $[H\setminus(n+1)]^n$ for $n\in H$; that is, iff $\I$ satisfies the simultaneous
almost-homogeneity of \cite[Prop.~3.15]{CDU2026} with \emph{prescribed identity cuts}. The
exact delta between Ramsey and $\mathscr S$-Ramsey is thus who controls the cuts.
\end{proposition}

\begin{proof}
Each $s\in\mathscr S$ equals $\{n\}\cup t$ with $n=\min s$ and
$t\in[\omega\setminus(n+1)]^{n}$, so colourings $c:\mathscr S\upharpoonright X\to2$
correspond exactly to sequences of sections $c_n:[X\setminus(n+1)]^n\to2$ ($n\in X$) via
$c_n(t)=c(\{n\}\cup t)$.

Suppose $\I$ is $\mathscr S$-Ramsey and let $(c_n)$ be given. Assemble the sections into
a colouring $c$ of $\mathscr S\upharpoonright X$ and take a positive homogeneous
$H\subseteq X$, of colour $i$ say. For $n\in H$ and $t\in[H\setminus(n+1)]^n$ we have
$\{n\}\cup t\in\mathscr S\upharpoonright H$, so $c_n(t)=i$: each section with index in
$H$ is constant on the prescribed tail.

Conversely, suppose the displayed property holds and let $c$ colour
$\mathscr S\upharpoonright X$. Take $H\in\Ipos$, $H\subseteq X$, such that each $c_n$
with $n\in H$ is constant on $[H\setminus(n+1)]^n$, with value $i_n\in2$. One of
$\{n\in H:i_n=0\}$ and $\{n\in H:i_n=1\}$ is positive; call it $H'$. Restricting to $H'$
shrinks the tails, so for $s\in\mathscr S\upharpoonright H'$ we get
$c(s)=c_{\min s}(s\setminus\{\min s\})=i_{\min s}$, and $i_n$ is constant for
$n\in H'$: $H'$ is homogeneous. Thus $\I$ is $\mathscr S$-Ramsey. The displayed property
is precisely the simultaneous almost-homogeneity of \cite[Prop.~3.15]{CDU2026} with the
cut at $n$ prescribed to be $n+1$ itself.
\end{proof}

\begin{remark}[why the coanalytic example escapes]\label{rem:shared}
In each of the three questions the natural definable family --- deciding antichains,
refining decompositions, barriers --- is coanalytic, not analytic: ``is a barrier'' is
$\Pi^1_1$-complete \cite[Prop.~4.45]{CDU2026}, and by Theorem~\ref{thm:width} no such
family can be a countable positive $\pi$-base. Borelness of $\I$ supplies no evident
mechanism to collapse these families to analytic sets, or to size $<\mathfrak c$, while
preserving their reach after forcing: a definable family repopulates
(Remark~\ref{rem:absolute-caveat}). This is exactly how the coanalytic example survives
--- $\Pi^1_1$ sets carry unbounded ranks in ZFC --- and why an affirmative answer to any
of the three questions must route through a boundedness or reflection theorem that uses
Borelness essentially, in the form crystallized by Remark~\ref{rem:unify}: Borelness must
force a properness-like reading of names on a positive condition.
\end{remark}

\appendix
\section{The added real is completion-strict}\label{app:completion}

This appendix explains why ``make the added real of low complexity'' does not open (DW).
Its content is known in substance --- a re-derivation, in Boolean/forcing language, of
facts equivalent to \cite[Thm~3.13]{HMTU2017} --- and we present it as such; the
contribution is the reformulation, which closes a family of approaches and locates the
added real precisely (Remark~\ref{rem:subexp}). Throughout, $\I$ is a tall Borel Ramsey
ideal, as in Section~\ref{sec:core}. Write $\mathbb B=\mathcal P(\omega)/\I$, a Boolean
algebra whose nonzero elements are the conditions of $\QI$, and let $\dot U$ be the
generic ultrafilter. Call a name $\dot w$ for an element of $2^\omega$
\emph{$\mathbb B$-represented} if there is a ground-model sequence $(C_n)_n$ of subsets of
$\omega$ with $\dot w(n)=1\iff C_n\in\dot U$ --- i.e.\ every bit's Boolean value is
realised by an actual set, $[\![\dot w(n)=1]\!]=[C_n]\in\mathbb B$ (not merely in the
completion $\mathrm{RO}(\mathbb B)$).

\begin{proposition}[$\mathbb B$-represented new real $\Rightarrow\conv$; known in
substance, cf.\ {\cite[Thm~3.13]{HMTU2017}}]\label{prop:Brep}
Let $\I$ be any tall ideal. If some $\mathbb B$-represented name $\dot w$ is forced to be a new
real, then $\conv\leK\I$; in particular $\I$ is not Ramsey.
\end{proposition}

\begin{proof}
For $s\in2^n$ put $A_s=\bigcap_{i<n}C_i^{s(i)}$ ($C^1=C$, $C^0=\omega\setminus C$). For fixed
$n$ the $A_s$ ($s\in2^n$) partition $\omega$ into $2^n$ pieces; the non-$\I$ ones form
$\mathcal P_n$, a finite partition into positive sets, hence a decomposition (a finite union
detects $\I$). As $A_{s^\frown i}\subseteq A_s$, $\mathcal P_{n+1}$ refines $\mathcal P_n$. A
decreasing branch $(A_{x\restriction n})_n$, $x\in2^\omega$, with positive pseudo-intersection
$Y$ has $Y\le_\I C_i^{x(i)}$ for all $i$, so $Y\Vdash\dot w=x\in V$; since $\dot w$ is forced
new, no positive condition forces $\dot w\in V$, so $Y\in\I$. All branch pseudo-intersections
are $\I$-small; \textup{(F2)} gives $\conv\leK\I$, and \textup{(F1)} kills Ramsey.
\end{proof}

\begin{remark}[$\mathbb B$-representation is exactly $(\omega,2)$-distributivity, hence
vacuous here]\label{rem:Brep-vacuous}
A $\mathbb B$-represented new real is precisely a witness that $\mathbb B$ fails
$(\omega,2)$-distributivity as a Boolean algebra: the itinerary
$\dot w=(\mathbf 1[C_n\in\dot U])_n$ is new iff the family $\{[C_n]\}$ has no positive common
decider, i.e.\ no positive $X$ with $X\le_\I C_n$ or $X\le_\I\omega\setminus C_n$ for all $n$.
But a tall Ramsey $\I$ \emph{is} $(\omega,2)$-distributive in the sense of
\cite[Thm~3.13]{HMTU2017} (below every positive set some positive set decides every countable
ground-model family). Hence for a Ramsey candidate \emph{every $\mathbb B$-represented
itinerary real is forced into $V$}; Proposition~\ref{prop:Brep} never fires. The real
guaranteed by \textup{(F3)} is therefore \emph{not} $\mathbb B$-represented: all its bits'
Boolean values lie in $\mathrm{RO}(\mathbb B)\setminus\mathbb B$.
\end{remark}

\begin{proposition}[the geometric generic point is tamed]\label{prop:geom}
Normalise $\I\subseteq\nwd(\Q)$ with $\Q$ dense in $[0,1]$ (Corollary~\ref{cor:normal}). Fix
a countable basis $(J_k)$ of half-open intervals of $[0,1]$ and let
$D_k=\{q\in\Q:q\in J_k\}$. The ``itinerary of the generic point''
$\dot w=(\mathbf 1[D_k\in\dot U])_k$ is $\mathbb B$-represented, hence \textup{(}by
Remark~\ref{rem:Brep-vacuous}\textup{)} forced into $V$ for a Ramsey $\I$. Thus the quotient
adds \emph{no} new real through the convergent-sequence geometry, \emph{even when} the limit
spectrum $L_\I$ (Proposition~\ref{prop:conv}) contains a perfect set: the generic limit point
is a ground-model point, decided on a dense set of positive convergent conditions.
\end{proposition}

\begin{proof}
The $D_k$ are ground-model sets and $\dot w(k)=1\iff D_k\in\dot U$, so $\dot w$ is
$\mathbb B$-represented by $(D_k)_k$, and Remark~\ref{rem:Brep-vacuous} applies verbatim:
by $(\omega,2)$-distributivity, below every positive set there is a positive $B$ deciding
every $D_k$ modulo $\I$, and such a $B$ forces $\dot w$ into $V$. These deciding
conditions are dense; below any of them the $h$-FinBW property in the reduced
presentation (\S\ref{sec:reduced}) supplies a further positive \emph{convergent}
condition, which still decides every $D_k$, hence the whole itinerary --- and with it the
limit point of the generic --- to be a ground-model object.
\end{proof}

\begin{remark}[the wrong parameter]\label{rem:subexp}
The taming in Remark~\ref{rem:Brep-vacuous} and Proposition~\ref{prop:geom} is effected by
$(\omega,2)$-distributivity, a property the candidate has. What would make the deciding
antichains of a name \emph{countable below a positive set} is \emph{properness of $\QI$ for
that name} --- and by Theorem~\ref{thm:reduction} this fails not only globally but below every
positive set, for every new-real name. The complexity of the individual added real is the
wrong parameter: a complete embedding of a ccc forcing (a Cohen real, say) into $\QI$ does not
bound the antichains of the non-proper host $\QI$, and the reals that \emph{would} come with
countable/finite antichains (the $\mathbb B$-represented ones, the geometric point) are exactly
the ones forced into $V$. The one genuinely new by-product is the crisp locus
$\mathrm{RO}(\mathbb B)\setminus\mathbb B$ for the added real, which tells future attempts to
target the completion, not the algebra.
\end{remark}

\section*{Disclosure of AI use}

The author used generative artificial-intelligence tools to assist with translating and
rewriting proofs, improving the overall structure of the paper, revising prose and grammar,
and refactoring and organizing the manuscript. All mathematical results and the final version
of the article are the sole responsibility of the author.



\begin{thebibliography}{99}

\bibitem{CDU2026}
J.~C. Cano, C.~A. Di Prisco, C. Uzc\'ategui-Aylwin,
\emph{Combinatorics of Ramsey ideals}, arXiv:2606.25477 (24 June 2026).

\bibitem{Farah1998}
I. Farah, \emph{Semiselective coideals}, Mathematika 45 (1998), 79--103.

\bibitem{GrebikHrusak2020}
J. Greb\'ik, M. Hru\v{s}\'ak, \emph{No minimal tall Borel ideal in the Katětov order},
Fund. Math. 248 (2020), 135--145.

\bibitem{GrebikUzcategui2019}
J. Greb\'ik, C. Uzc\'ategui, \emph{Bases and Borel selectors for tall families},
J. Symbolic Logic 84 (2019), 359--375.

\bibitem{GrebikVidnyanszky2023}
J. Greb\'ik, Z. Vidny\'anszky, \emph{Tall $F_\sigma$ subideals of tall analytic ideals},
Proc. Amer. Math. Soc. 151 (2023), 4043--4046.

\bibitem{HrusakKatetov2017}
M. Hru\v{s}\'ak, \emph{Katětov order on Borel ideals}, Arch. Math. Logic 56 (2017),
831--847.

\bibitem{HrusakSurvey}
M. Hru\v{s}\'ak, \emph{Ideals on countable sets: a survey with questions},
arXiv:1902.08677.

\bibitem{Kechris1995}
A. S. Kechris, \emph{Classical Descriptive Set Theory}, Graduate Texts in Mathematics 156,
Springer, 1995.

\bibitem{HMTU2017}
M. Hru\v{s}\'ak, D. Meza-Alc\'antara, E. Th\"ummel, C. Uzc\'ategui,
\emph{Ramsey type properties of ideals}, Ann. Pure Appl. Logic 168 (2017), 2022--2049.

\bibitem{KwelaSabok2015}
A. Kwela, M. Sabok, \emph{Topological representations}, J. Math. Anal. Appl. 422 (2015),
1434--1446.

\bibitem{Martin1975}
D. A. Martin, \emph{Borel determinacy}, Ann. of Math. (2) 102 (1975), 363--371.

\bibitem{PelayoGomez}
J. de J. Pelayo-G\'omez, \emph{Infinite games and Ramsey properties of $F_\sigma$ ideals},
J. Symbolic Logic (2025), 1--25,
doi:10.1017/jsl.2025.12; arXiv:1808.09088.

\bibitem{Solecki1994}
S. Solecki, \emph{Covering analytic sets by families of closed sets},
J. Symbolic Logic 59 (1994), 1022--1031.

\end{thebibliography}
\end{document}